\documentclass[reqno]{amsart}
\usepackage{graphicx} 

\usepackage{mathtools}
\mathtoolsset{showonlyrefs}

\usepackage{amsmath,amssymb,amsfonts,amsthm}
\usepackage{fancyvrb}
\usepackage{inconsolata}

\usepackage[mathscr]{eucal}
\usepackage{bbm}
\usepackage{array}
\usepackage{mathdots}
\usepackage{ytableau}
\usepackage{longtable}
\usepackage{nth}

\usepackage{xcolor}

\usepackage{hyperref}
\newcommand\myshade{85}
\colorlet{mylinkcolor}{red}
\colorlet{mycitecolor}{blue}
\colorlet{myurlcolor}{magenta}
\hypersetup{
	linkcolor  = mylinkcolor,
	citecolor  = mycitecolor,
	urlcolor   = myurlcolor!\myshade!black,
	colorlinks = true,
}

\usepackage{tikz}

\usetikzlibrary{arrows,math,calc,shapes.geometric}
\usepackage[all,cmtip]{xy}

\usepackage{microtype}

\usepackage{enumerate}
\usepackage[shortlabels]{enumitem}

\usepackage{etoolbox}
\patchcmd{\section}{\scshape}{\bfseries}{}{}
\makeatletter
\renewcommand{\@secnumfont}{\bfseries}
\makeatother

\theoremstyle{definition}
\newtheorem{thm}{Theorem}
\newtheorem*{thm*}{Theorem}
\newtheorem{prop}[thm]{Proposition} 

\newtheorem{defi}[thm]{Definition} 

\newtheorem{rem}[thm]{Remark}
\newtheorem{ex}[thm]{Example}

\newtheorem*{quest*}{Question}

\newcommand{\ala}[1]{\widehat{\mathfrak{#1}}}

\newcommand{\ZZ}{\mathbb{Z}}

\newcommand{\ds}{\displaystyle}

\newcommand{\rel}{\text{rel}}

\newcommand{\vv}{{\vec{v}}}
\newcommand{\vn}{{\vec{n}}}

\newcommand{\vz}{{\vec{0}}}

\title{Tight cylindric partitions}
\author{Shashank Kanade}
\address{Department of Mathematics, University of Denver, Denver, CO 80208}
\email{shashank.kanade@du.edu}

\author{Matthew C.\ Russell}
\address{Department of Statistics, Texas A\&M University, College Station, TX 77843}
\email{matthewcrussell@tamu.edu}

\allowdisplaybreaks

\begin{document}
\begin{abstract}
In this note, we initiate the study of generating functions for tight cylindric partitions. 
For general (i.e., $r$-rowed for $r\geq 2$) tight cylindric partitions, we provide analogs of the Corteel--Welsh functional equations. We prove closed forms for the bivariate generating functions for 2-rowed tight cylindric partitions. We also show that these partitions are in bijection with a class of partitions studied by Dousse, Hardiman, and Konan. 
\end{abstract}
\maketitle

\section{Overview}

It is well known (\cite{Tin-three}, \cite{FodWel}, etc.) that the univariate generating function of $r$-rowed cylindric partitions of a profile $c$ is related to the principally specialized character of the integrable $\ala{sl}_r$ module of highest weight $\Lambda(c)$, up to a factor of ${(q^r;q^r)_\infty^{-1}}$:
\begin{align}
	\mathrm{g.f.} \left( \mathscr{C}_c \right) = {(q^r;q^r)_\infty^{-1}}\cdot  {\mathrm{Pr}\left( L_{\ala{sl}_r}(\Lambda(c))\right)}.
	\label{eqn:cylchar}
\end{align}

Cylindric partitions with $r$-rows and a fixed profile  can be thought of as restrictions of skew plane partitions to $r$-rows, along with certain periodic ``boundary conditions''. These were first studied by Gessel and Krattenthaler \cite{GesKra} who used the Gessel--Viennot--Lindstr\"om principle to get determinantal formulas for their univariate generating functions. These formulas are much more general and involve a lot of flexibility.
Importantly, Foda and Welsh \cite{FodWel} recognized that simply bounding the largest part in the cylindric partitions to be less than $a$, using the Gessel--Krattenthaler determinantal formula, taking the limit as $a\rightarrow\infty$, and finally using the type $\mathrm{A}_{r-1}^{(1)}$ Macdonald identity yields a beautiful product form for the univariate generating function. This product form was also found by Borodin \cite{Bor} in the context of integrable probability.
In any case, the lesson is that the ``maximum part'' statistic also plays an important role in the theory of cylindric partitions. We shall use the variable $z$ to keep track of this statistic.

Indeed, Corteel and Welsh \cite{CorWel} found functional equations for the bivariate generating functions for cylindric partitions. These equations have proved very useful in proving closed form formulas for the bivariate (and thus univariate), generating functions of cylindric partitions (see \cite{CorWel}, \cite{CorDouUnc}, \cite{KanRus-cyl}, etc.). While the case of two rows is fairly straightforward (see Remark \ref{rem:allcyl2rowzq} below), the case of three rows is already very intriguing and the corresponding bivariate conjectural formulas remain open (except for some small profiles) \cite{KanRus-cyl}.
It is natural that refining the generating functions with more statistics is an important yet difficult process.

As mentioned above, our interest in cylindric partitions stems from the fact that they are intricately linked to the combinatorics of standard (i.e., integrable highest-weight) modules for $\ala{sl}_r$. This connection goes back to the work of the ``Kyoto School'' \cite{DatJimKunMiwOka}, \cite{JimMisMiwOka}, explained further in \cite{Tin-three} and \cite{FodWel}.
Yet, as is visible in \eqref{eqn:cylchar}, the univariate generating function of cylindric partitions differs from principally specialized characters by a somewhat undesirable factor.
The reason for this phenomenon appears to be the fact that the $\ala{sl}_r$ crystal graph on the cylindric partitions of a fixed profile is not connected \cite{Tin-three}.

Thankfully, the objects that directly correspond to $\mathrm{Pr}\left(L_{\ala{sl}_r}(\Lambda(c))\right)$ are the \emph{tight} or \emph{highest lift} cylindric partitions,
i.e. (see \cite[Eq.\ (33)]{FodWel}),
\begin{align}
	\mathrm{g.f.} \left( \mathscr{T}_{c} \right) = {\mathrm{Pr}\left( L_{\ala{sl}_r}(\Lambda(c))\right)}.
	\label{eqn:tightcylchar}
\end{align}
The reason these are called ``tight'' is that, when modeling cylindric partitions by abaci, tight cylindric partitions are precisely those whose yokes are as close together as possible \cite[Sec.\ 4.7 and footnote 14]{FodWel}.

The aim of this note is to initiate the study of bivariate generating functions of the tight cylindric partitions by  focusing first on the $2$-rowed case. We achieve the following:
\begin{enumerate}
	\item Analogs of the Corteel--Welsh \cite{CorWel} functional equations for tight cylindric partitions (Theorem \ref{thm:CWrec}).
	\item A description and proof of the  closed-form formulas for the bivariate generating functions for $2$-rowed tight cylindric partitions (Theorem \ref{thm:bivar2rowtight}).  
	\item Alternate proofs of a combinatorial identity of Dousse, Hardiman, and Konan \cite{DouHarKon} using the \eqref{eqn:tightcylchar} for 2-rowed tight cylindric partitions. This is not surprising, as both the objects are governed by $\ala{sl}_2$ crystals 
    (Section \ref{sec:DHKpartitions}).
\end{enumerate}

\subsection*{Acknowledgments}
SK gratefully acknowledges the support from Simons Travel Support for Mathematicians. 
We thank S.\ Ole Warnaar for many illuminating discussions. The survey \cite{MSZsurvey} was very helpful in some preliminary investigations.

\section{Basic definitions}

We start by recalling the definition of cylindric partitions.

\begin{defi} \label{def:cylp}
Given $c=(c_1,\cdots,c_r)\in \ZZ_{\geq 0}^r$, a cylindric partition of profile $c$ is an $r$-tuple of (possibly empty) partitions $\pi=(\pi^{(1)},\cdots, \pi^{(r)})$ such that for all $i,j$:
\begin{enumerate}
\item $\pi^{(i)}_j\geq \pi^{(i+1)}_{j+c_{(i+1)}}$
\item $\pi^{(k)}_j\geq \pi^{(1)}_{j+c_1}$.
\end{enumerate}
Let $\mathscr{C}_c$ be the set of cylindric partitions of a profile $c$.
\end{defi} 
\begin{defi}\label{def:ranklevelhtwt}
Given a profile $c=(c_1,\cdots, c_r)$ we have the corresponding dominant integral highest weight for $\ala{sl}_r$ given as:
\begin{align}
    \Lambda(c) = c_r\Lambda_0 + c_1\Lambda_1 + \cdots +c_{r-1}\Lambda_{r-1}.
    \label{eqn:domwt}
\end{align}
We thus say that $r$ is the rank and $\ell=c_1+\cdots+c_r$ is the level of $c$.
\end{defi}    
\begin{defi}\label{def:tightcylp}
	We say that a cylindric partition $\pi=(\pi^{(1)},\cdots, \pi^{(r)})$ of profile $c$ is \emph{tight} if for every $j\in \ZZ_{>0}$, there exists a $\pi^{(i)}$ that does not contain $j$ as a part. 
	Let $\mathscr{T}_c$ be the set of tight cylindric partitions of profile $c$.
\end{defi}
	
Let $\mathrm{max}(\pi)$ denote the largest part of a cylindric partition $\pi$ and $\mathrm{wt}(\pi)$ denote the sum of all entries of $\pi$. We let the generating functions of all, respectively, tight cylindric partitions of profile $c$ be:
\begin{align}
C_c(z,q) = \sum_{\pi \in \mathscr{C}_c} z^{\mathrm{max}(\pi)}q^{\mathrm{wt}(\pi)},\quad\quad
T_c(z,q) = \sum_{\pi \in \mathscr{T}_c} z^{\mathrm{max}(\pi)}q^{\mathrm{wt}(\pi)}.
\end{align}

\begin{thm}
	\label{thm:tightprod}
	For a profile $c=(c_1,c_2,\cdots, c_r)$, we have that: 
	\begin{align}
		&T_c(1,q) = \mathrm{Pr}(L_{\ala{sl}_r}(\Lambda(c)))
        &=
        \dfrac{(q^r;q^r)_\infty(q^m;q^m)_\infty^{r-1}}{(q)_\infty^r}
        \prod_{1\leq i<j\leq r}
        \theta(q^{j-i+c_i+\cdots+c_{j-1}};q^m)
	\end{align}
	where $m=r+\ell$, $\Lambda(c)$ is as in \eqref{eqn:domwt}, and $\theta(a;q)=(a;q)_\infty(q/a;q)_\infty$.
\end{thm}
\begin{proof}
For the first equality, see \cite[Eq.\ (33) and App.\ C]{FodWel}.
The fact that $\mathrm{Pr}(L_{\ala{sl}_r}(\Lambda(c)))$ is an infinite periodic product is due to \cite{Lep-numer}. The actual evaluation is standard. 
\end{proof}

Note that, for 2-rowed profiles, the infinite product here is $\left(-q;q\right)_\infty$ times the generating function for the Gordon-Andrews or Andrews-Bressoud (depending on the parity of $\ell$) identities, with $\ell\geq 1$ and $0\leq b\leq \ell$:
\begin{equation}
T_{(b,\ell-b)}(1,q) = T_{(\ell-b,b)}(1,q) =
        \dfrac{(-q;q)_\infty\left(q^{b+1},q^{\ell-b+1},q^{\ell+2};q^{\ell+2}\right)_\infty}{(q;q)_\infty}.
\end{equation}

We now describe the $2$-stringed abaci, closely following \cite{FodWel}. Such an abacus is an infinite sequence of sites arranged in 2 strings, with beads and vacancies at each site.
To the right of any site, there are only finitely many beads, and to the left of any site, there are only finitely many vacancies. The beads are yoked together vertically, no yoke has a negative slope, yokes do not cross, and each bead belongs to a unique yoke. After moving sufficiently left, the repeating pattern of yokes has a constant shape. 

Fix $\ell\geq 1$ and $0\leq a\leq \ell$.
We say that a yoke has shape $a$ if the top bead is $a$ places to the right of the bottom bead. We say that an abacus is of type $\mathcal{A}_{\ell-a,a}$  if only yokes of shape $0,1,2,\cdots,\ell$ appear in the abacus, and infinitely to the left, only yokes of shape $a$ appear.

We now associate an abacus of type $\mathcal{A}_{\ell-a,a}$ to a cylindric partition of profile $(\ell-a,a)$ as follows. 
The $s^{\text{th}}$ string of the abacus corresponds to the $s^{\text{th}}$ row of the cylindric partition. Scanning the $s^{\text{th}}$ row of cylindric partition from left to right, the $i^{\text{th}}$ entry gives the number of beads that appear on the $s^{\text{th}}$ string to the right of the $i^{\text{th}}$ vacancy (counting the vacancies from left to right).

The right-most yoke belonging to the infinite sea of repeating yokes (of shape $a$) is referred to as the zeroth yoke, and we number the yokes successively to the right.

\begin{ex}
\label{ex:tightcylp}
Our working example is the following cylindric partition $\pi=(\pi^{(1)},\pi^{(2)})$ of profile $(3,0)$.
\begin{align*}
\begin{matrix}
\textcolor{gray}{\pi^{(2)} \rightarrow } & & & &\textcolor{gray}{9} & \textcolor{gray}{5} & \textcolor{gray}{1} \\
\pi^{(1)} \rightarrow & 10 & 8 & 4 & 3 &3 \\
\pi^{(2)} \rightarrow &  9& 5 & 1
\end{matrix}
\end{align*}
which corresponds to the following abacus of type $\mathcal{A}_{3,0}$:
\begin{center}
\begin{tikzpicture}[scale=0.5]
    \draw[dotted] (-2.75,1) -- (17.75,1);
    \draw[dotted] (-2.75,0) -- (17.75,0);
	\node at (-3,0.5) {$\cdots$};
	
	\node at (-2,1) {$\bullet$};
	\node at (-1,1) {$\bullet$};
	
	\node at (0,1) {$\bullet$};
	\node[fill=white,inner sep=0] at (1,1) {$\circ$};
	\node at (2,1) {$\bullet$};
	\node at (3,1) {$\bullet$};
	\node[fill=white,inner sep=0] at (4,1) {$\circ$};
	\node at (5,1) {$\bullet$};
	\node at (6,1) {$\bullet$};
	\node at (7,1) {$\bullet$};
	\node at (8,1) {$\bullet$};
	\node[fill=white,inner sep=0] at (9,1) {$\circ$};
	\node at (10,1) {$\bullet$};
	\node[fill=white,inner sep=0] at (11,1) {$\circ$};
	\node[fill=white,inner sep=0] at (12,1) {$\circ$};
	\node at (13,1) {$\bullet$};
	\node at (14,1) {$\bullet$};
	\node at (15,1) {$\bullet$};

	\node at (-2,0) {$\bullet$};
	\node at (-1,0) {$\bullet$};
	
	\node at (0,0) {$\bullet$};
	\node at (1,0) {$\bullet$};
	\node[fill=white,inner sep=0] at (2,0) {$\circ$};
	\node at (3,0) {$\bullet$};
	\node at (4,0) {$\bullet$};
	\node at (5,0) {$\bullet$};
	\node at (6,0) {$\bullet$};
	\node[fill=white,inner sep=0] at (7,0) {$\circ$};
	\node at (8,0) {$\bullet$};
	\node at (9,0) {$\bullet$};
	\node at (10,0) {$\bullet$};
	\node at (11,0) {$\bullet$};
	\node[fill=white,inner sep=0] at (12,0) {$\circ$};
	\node at (13,0) {$\bullet$};
	
	\node at (0,-0.5) {$\scriptstyle{0}$};
	\node at (1,-0.5) {$\scriptstyle{1}$};
	\node at (2,-0.5) {$\scriptstyle{2}$};
	\node at (3,-0.5) {$\scriptstyle{3}$};
	\node at (4,-0.5) {$\scriptstyle{4}$};
	\node at (5,-0.5) {$\scriptstyle{5}$};
	\node at (6,-0.5) {$\scriptstyle{6}$};
	\node at (7,-0.5) {$\scriptstyle{7}$};
	\node at (8,-0.5) {$\scriptstyle{8}$};
	\node at (9,-0.5) {$\scriptstyle{9}$};
	\node at (10,-0.5) {$\scriptstyle{10}$};
	\node at (11,-0.5) {$\scriptstyle{11}$};
	\node at (12,-0.5) {$\scriptstyle{12}$};
	\node at (13,-0.5) {$\scriptstyle{13}$};
	\node at (14,-0.5) {$\scriptstyle{14}$};
	\node at (15,-0.5) {$\scriptstyle{15}$};

    \node at (-5,1) {\scriptsize string $1 \rightarrow$};
    \node at (-5,0) {\scriptsize string $2 \rightarrow$};
    
	\draw (-2,0) -- (-2,1);
	\draw (-1,0) -- (-1,1);
	\draw (0,0) -- (0,1);
	\draw (1,0) -- (2,1);
	\draw (3,0) -- (3,1);
	\draw (4,0) -- (5,1);
	\draw (5,0) -- (6,1);
	\draw (6,0) -- (7,1);
	\draw (8,0) -- (8,1);
	\draw (9,0) -- (10,1);
	\draw (10,0) -- (13,1);
	\draw (11,0) -- (14,1);
	\draw (13,0) -- (15,1);
	
	\node at (14,0) {$\circ$};
	\node at (15,0) {$\circ$};
	\node at (16,0) {$\circ$};
	\node at (16,1) {$\circ$};
	\node at (17,0) {$\circ$};
	\node at (17,1) {$\circ$};
	
	\node at (18,0.5) {$\cdots$};
	
\end{tikzpicture}
\end{center}

In this abacus, the zeroth yoke connects the sites at location $0$ on both the strings, and has shape $0$.  To the left of this yoke, we have an infinite sea of repeating yokes of shape $0$. To the right, the
first yoke has shape $1$ and joins the bead at site $2$ on string $1$ with the bead at site $1$ on string $2$. Continuing,  there are a total of $10$ yokes to the right of the zeroth yoke, and the last, i.e., the $10^{\text{th}}$, yoke connects the beads at site $15$ on string $1$ and site $13$ on string $2$. To the right of the $10^{\text{th}}$ yoke, we have an infinite sea of vacancies on each string.
\end{ex}

\begin{defi}
	A $2$-stringed abacus is called {\it tight} if none of the yokes could be shifted to the left. 
\end{defi}

\begin{thm}(\cite[Sec. 4.7, 4.8, App.\ C]{FodWel})
	A $2$-rowed cylindric partition is tight if and only if the corresponding abacus is tight. 
\end{thm}

\begin{ex}
The cylindric partition and the corresponding abacus in Example \ref{ex:tightcylp} are tight.
\end{ex}

\begin{rem}
We have only recalled the $2$-stringed abaci here, but this works more generally for $r$-stringed abaci \cite{FodWel}. The conventions for abaci in \cite{Tin-three} are different.
\end{rem}

\section{Functional Equations}
\subsection{Correcting the Corteel--Welsh recurrences}

For general $r$-rowed tight cylindric partitions of profile $c=(c_1,c_2,\cdots, c_{r})$, the following recurrences hold.
\begin{thm}
\label{thm:CWrec}
Let $c=(c_1,\cdots,c_r)$ with $r\geq 2$ and $\ell=c_1+\cdots+c_r\geq 1$.
Let $I_c \subseteq \{1,2,\cdots,r\}$, where $i\in I_c$ if and only if $c_i>0$. 
For a subset $J\subseteq I_c$, define the composition $c(J)$ as in \cite[Sec.\ 3]{CorWel}:
\begin{align}
c_i(J)=\begin{cases}
c_i-1 & i\in J, (i-1)\not\in J\\
c_i+1 & i\not\in J, (i-1)\in J\\
c_i & \mathrm{otherwise}
\end{cases}
\end{align}
where $c_0=c_r$.
Then,
\begin{align}
T_{c}(z,q) + \dfrac{zq^{r}}{1-zq^r}T_c(zq^r,q)= \sum_{\emptyset \neq J\subseteq I_c} (-1)^{|J|-1}\dfrac{T_{c(J)}{(zq^{|J|},q)}}{1-zq^{|J|}}
\label{eqn:tightCWrec}
\end{align}
\end{thm}

\begin{proof}

We follow the general process of \cite[Prop.\ 3.1]{CorWel} adapting it to the situation of tight cylindric partitions as needed; however, some subtleties emerge.

Fix a profile $\widetilde{c}$ with $r$ rows. Let $\mathscr{P}_{\widetilde{c}}\subseteq \mathscr{C}_{\widetilde{c}}$ be \emph{some} subset  of cylindric partitions of this profile (for instance, the tight cylindric partitions), and let $P_{\widetilde{c}}(z,q)$ be its generating function.
Now, for fixed $t\geq 0, j\geq 1$, we interpret $(zq^j)^t P_{\widetilde{c}}(zq^j;q)$.
Given a cylindric partition $\widetilde{\pi}\in \mathscr{P}_{\widetilde{c}}$ consider a new cylindric partition, say $\pi$, of a possibly different profile $c$ constructed as follows. Set $m=\mathrm{max}(\widetilde{\pi})$, and create $j$ new parts of weight $m+t$ and attach them \emph{somewhere} (to be determined precisely later) to $\widetilde{\pi}$, thus obtaining $\pi$. Now, $\mathrm{max}(\pi)=m+t$ and $\mathrm{wt}(\pi)=\mathrm{wt}(\widetilde{\pi}) + jm+jt$. Doing this 
to all $\widetilde{\pi}$, and then calculating the generating function of $\pi$ thus obtained, we have:
\begin{align}
\sum_{\widetilde{\pi}\in \mathscr{P}_{\widetilde{c}}} z^{\mathrm{max}(\widetilde{\pi})+t}q^{\mathrm{wt}(\widetilde{\pi})+j\mathrm{max}(\widetilde{\pi})+jt}
 =(zq^j)^t P_{\widetilde{c}}(zq^j,q).
 \label{eqn:addboxes}
\end{align}

The left side of \eqref{eqn:tightCWrec} corresponds to the cylindric partitions $\pi$ of profile $c$ such that either:
\begin{enumerate}
\item $\pi$ is tight, or
\item With $\pi=(\pi^{(1)}, \pi^{(2)},\cdots, \pi^{(r)})$, we have:
\begin{enumerate}
\item $\mathrm{max}(\pi)=\pi^{(1)}_1=\pi^{(2)}_1=\cdots = \pi^{(r)}_1$ (thus, these partitions are \emph{not} tight), 
\item these are the only places where $\mathrm{max}(\pi)$ appears (i.e, $\pi^{(i)}_1 > \pi^{(i)}_2$ for all $i$), and finally, 
\item removing all parts equal to $\max(\pi)$ gets us a tight cylindric partition (say, $\widetilde{\pi}$) of profile $c$.
\end{enumerate}
\end{enumerate}
We let the set of cylindric partitions satisfying the second condition be denoted by $\mathscr{P}_c$. Its generating function is 
$\frac{zq^r}{1-zq^r}T_c(zq^r;q)$
by following the process of \eqref{eqn:addboxes} for all $t\geq 1$, while adding the newly created parts at the beginning of each partition of $\widetilde{\pi}$ of profile $c$. 

Now, exactly as in \cite[Prop.\ 3.1]{CorWel}, for each $\emptyset \neq J \subseteq I_c$, we perform the procedure encapsulated in \eqref{eqn:addboxes} on \emph{tight} cylindric partitions $\pi=(\pi^{(1)},\cdots,\pi^{(r)})$ of profile $c(J)$ by attaching the newly created parts at the beginning of $\pi^{(j)}$ with $j\in J$. The resulting partitions are cylindric (but not necessarily tight) of profile $c$, and have the generating function $(zq^{|J|})^t\,T_{c(J)}(zq^{|J|};q)$. Summing over $t\geq 0$, we get $T_{c(J)}(zq^{|J|};q)/(1-zq^{|J|}).$

Our hope would have been that this process yields exactly the cylindric partitions appearing in $\mathscr{T}_c\sqcup \mathscr{P}_c$ (which are counted by the left side of \eqref{eqn:tightCWrec}), and that in the signed sum, each of these partitions gets counted exactly once. However, a priori, the right side of \eqref{eqn:tightCWrec} generates a larger class of partitions, and the extraneous partitions disappear when the signed sum is considered. We now substantiate these claims.

Exactly as in \cite{CorWel}, for $\pi\in \mathscr{C}_c$, since $\pi^{(i-1)}_1 \geq  \pi^{(i)}_1$
whenever $i \not\in I_c$ (we let $\pi^{(0)}=\pi^{(r)}$),
it must be the case that $\max(\pi)$ appears as $\pi^{(i)}_1$ for some $i\in I_c$. Let $I_\pi\subseteq I_{c}$ be the collection of all such $i$. (Note that if $\pi\in \mathscr{P}_c$ then $I_{\pi} = I_c$.) 
We refer to $\pi^{(i)}_1$ for $i\in I_\pi$ as the corner-max parts of $\pi$.

Now, let $\pi\in \mathscr{T}_c\sqcup \mathscr{P}_c$.
Now, it can be seen that $\pi$ appears in the summand corresponding to a $J$ on the right side of \eqref{eqn:tightCWrec} if and only if $\emptyset \neq J \subseteq I_\pi$.
Thus, the total contribution to $\pi$ in the signed sum on the right side is:
\begin{align*}
\sum_{\emptyset \neq J \subseteq I_\pi} (-1)^{|J|-1} = 1.
\end{align*}

Now suppose $\pi \not\in (\mathscr{T}_c \sqcup \mathscr{P}_c)$, yet appears in at least one summand on the right side of \eqref{eqn:tightCWrec}, say by adding some of the corner parts (corresponding to locations $\emptyset\neq J\subseteq I_c$) of $\pi$ to a certain tight partition $\widetilde{\pi}$ of profile $c(J)$.
Then, $\pi$ has the following properties:
\begin{enumerate}

\item Since $\widetilde{\pi}$ is tight but $\pi$ is not, it cannot be that we are adding strictly larger parts  than $\max(\widetilde{\pi})$ to obtain $\pi$. This means that new parts added to $\widetilde{\pi}$ must be equal to the $\max(\widetilde{\pi})=\max(\pi)$, and since this breaks tightness, it must also be that $\max(\pi)$ appears in each row of $\pi$. In particular, $\max(\pi)=\pi^{(i)}_1$ for all $i\in\{1,2,\cdots,r\}$.

\item Further, since $\pi \not\in \mathscr{P}_c$, it must be that $\max(\pi)=\pi^{(i)}_2$ for some $i$ as well. 
Now, removing all of the first parts $\pi^{(i)}_1$ for $i=1,2,\cdots,r$, we obtain another cylindric partition, say $\pi'$, of the same profile. Thus, again, we have the set of corner-max parts $\emptyset \neq I_{\pi'} \subseteq I_c$.

\end{enumerate}
At this point, we have a sequence of cylindric partitions $\pi' \subset \widetilde{\pi} \subset \pi$ (where $\subset$ corresponds to adding parts), with $\max(\pi')=\max(\widetilde{\pi})=\max(\pi)$.  It is helpful to keep Figure \ref{fig:tightCWproof} in mind as an example, where $m$ denotes the maximum part.

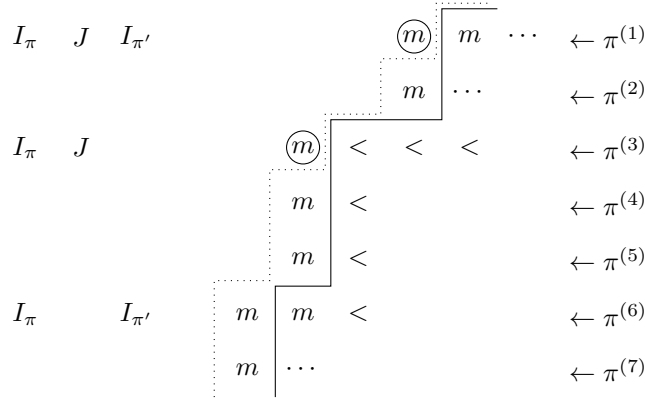
\begin{figure}
\begin{tikzpicture}[x=2.1em, y=2.1em]
\node[anchor=center,circle,draw, inner sep=0.1em] at (3,7) {$m$}; \node[anchor=center] at (4,7) {$m$}; \node[anchor=center] at (5,7) {$\cdots$};

\node[anchor=center] at (3,6) {$m$}; \node[anchor=center] at (4,6) {$\cdots$};

\node[anchor=center,inner sep=0.1em,circle,draw] at (1,5) {$m$};  \node[anchor=center] at (2,5) 
{$<$}; \node[anchor=center] at (3,5) {$<$}; \node[anchor=center] at (4,5) {$<$};

\node[anchor=center] at (1,4) {$m$};  \node[anchor=center] at (2,4) {$<$};

\node[anchor=center] at (1,3) {$m$};  \node[anchor=center] at (2,3) {$<$};
\node[anchor=center] at (0,2) {$m$}; \node[anchor=center] at (1,2) {$m$};  \node[anchor=center] at (2,2) {$<$}; 
\node[anchor=center] at (0,1) {$m$}; 
\node[anchor=center] at (1,1) {$\cdots$}; 
\draw[dotted] (-0.6,0.5) -- (-0.6,2.6) -- (0.4,2.6)--(0.4,4.6)-- (1.4,4.6) -- (1.4,5.6) --(2.4,5.6) --
(2.4,6.6) -- (3.4,6.6) --
(3.4,7.6)--(4.4,7.6);
\draw (0.5,0.5) -- (0.5,2.5) -- (1.5,2.5)--(1.5,5.5)--(3.5,5.5) -- (3.5,7.5)--(4.5,7.5);
\node at (-4,2) {$I_{\pi}$};
\node at (-4,5) {$I_{\pi}$};
\node at (-4,7) {$I_{\pi}$};
\node at (-3,5) {$J$};
\node at (-3,7) {$J$};
\node at (-2,2) {$I_{\pi'}$};
\node at (-2,7) {$I_{\pi'}$};
\node at (6.5,1) {$\leftarrow \pi^{(7)}$};
\node at (6.5,2) {$\leftarrow \pi^{(6)}$};
\node at (6.5,3) {$\leftarrow \pi^{(5)}$};
\node at (6.5,4) {$\leftarrow \pi^{(4)}$};
\node at (6.5,5) {$\leftarrow \pi^{(3)}$};
\node at (6.5,6) {$\leftarrow \pi^{(2)}$};
\node at (6.5,7) {$\leftarrow \pi^{(1)}$};
\end{tikzpicture}
\caption{An example of a cylindric partition $\pi$ is depicted. The partition to the right of dotted line is $\widetilde{\pi}$. The circled parts in rows marked $J$ are added to $\widetilde{\pi}$ to get $\pi$. The partition to the right of solid line, i.e., removing all first parts from $\pi$, is $\pi'$. The subset of $I_\pi$ of rows where $\pi'$ has the part $m$ is marked $I_{\pi'}$.  Parts marked $<$ are strictly less than $m$. One can now modify this picture and check that if $J\subseteq I_{\pi'}$, then $\widetilde{\pi}$ is not tight. All other choices of $\emptyset\neq J\subseteq I_\pi$ are allowed.}
\label{fig:tightCWproof}
\end{figure}

We now claim that $J\not\subseteq I_{\pi'}$. 
If $J\subseteq I_{\pi'}$, then, we have that:
\begin{enumerate}
\item $\widetilde{\pi}^{(i)}_1 = \pi_1^{(i)}=\max(\pi)$ for all $i\in \{1,2,\cdots,r\} \backslash J$ (this was already the case), 
\item For $i\in J$, $\widetilde{\pi}^{(i)}_1 = \pi^{(i)}_2$, but also, $\pi^{(i)}_2=(\pi')^{(i)}_1=\max(\pi)$ for all $i\in J\subseteq I_{\pi'}$,
\end{enumerate}
thus breaking the tightness of $\widetilde{\pi}$.

Thus, such a $\pi$ appears in the right side of \eqref{eqn:tightCWrec} for all $J$ such that $\emptyset\neq J\subseteq I_c$ with $J\not\subseteq I_{\pi'}$.
So, the number of times $\pi$ is counted in the alternating sum is:
\begin{align*}
\sum_{\substack{\emptyset\neq J\subseteq I_c \\ J\not\subseteq I_{\pi'}}} (-1)^{|J|-1}
=\left( \sum_{J\subseteq I_c} 
-\sum_{J\subseteq I_{\pi'}}
\right)(-1)^{|J|-1}=0-0=0.
\end{align*}
\end{proof}

In the $2$-rowed case, we record the recurrences explicitly as follows, and refer to them as the ``diamond relations''. \begin{prop}
    We have for each $\ell\geq 1$, $1\leq i\leq \ell-1$:
    \begin{align} 
        T_{(\ell-i,i)}(z,q) &+ \dfrac{zq^2}{1-zq^2}T_{(\ell-i,i)}(zq^2,q) \notag\\
        &= \dfrac{T_{(\ell-i+1,i-1)}(zq,q)}{1-zq}+\dfrac{T_{(\ell-i-1,i+1)}(zq,q)}{1-zq}
        -\dfrac{T_{(\ell-i,i)}(zq^2,q)}{1-zq^2},
        \label{eq:Rfunceqmain}
    \end{align}
    and we have:
    \begin{align} \label{eq:Rfunceqedge}
        T_{(\ell,0)}(z,q) + \dfrac{zq^2}{1-zq^2}T_{(\ell,0)}(zq^2,q) &= \dfrac{T_{(\ell-1,1)}(zq,q)}{1-zq},  \\
        T_{(0,\ell)}(z,q) + \dfrac{zq^2}{1-zq^2}T_{(0,\ell)}(zq^2,q) &= \dfrac{T_{(1,\ell-1)}(zq,q)}{1-zq}.
    \end{align}
\end{prop}

\subsection{Recurrences without inclusion--exclusion for \texorpdfstring{$2$}{2} rows}
We now provide a functional equation that will be used to get a connection with the Dousse--Hardiman--Konan partitions later.
\label{sec:tightrec}

\begin{prop}
Fix $\ell\in\ZZ_{\geq 1}$ and let $c_1+c_2=\ell$ ($c_1,c_2\in\ZZ_{\geq 0}$). Then:
\label{prop:2rowtightrec}
\begin{align}
	T_{(c_1,c_2)}(z,q) 
	&= 1+ \sum_{\substack{0\leq a\leq \ell,\\ a\neq c_2 }} \dfrac{zq^{|a-c_2|}}{1-zq^{|a-c_2|}}T_{(\ell-a,a)}(zq^{|a-c_2|},q).
    \label{eqn:tightrec2}
\end{align}
\end{prop}

\begin{proof}
Pick a tight cylindric partition $\pi$ with profile $(c_1,c_2)$, suppose that $\pi$ is non-empty, and let $m=\mathrm{max}(\pi)$.

For $i=1,2$, if $t_i$ is the number of parts equal to $m$ in $\pi^{(i)}$, then at least one of $t_1$ and $t_2$ must be $0$; otherwise, $m$ will appear in both rows.
We cannot have $t_2>c_2$; otherwise, the first part in $\pi^{(1)}$, which is stacked on top of the $(c_2+1)^\text{st}$ part of $\pi^{(2)}$ will be forced to be $m$, thus making $m$ appear in both $\pi^{(1)}$ and $\pi^{(2)}$.
By cylindricality, and similar reasoning, $t_1\leq c_1$. 

Removing the parts equal to $m$ leaves us with a (possibly empty) partition of profile $(c_1-t_1+t_2,  c_2-t_2+t_1)$, with largest part that is strictly less than $m$. Each of these profiles corresponds to a summand on the far right of \eqref{eqn:tightrec2} when we employ the process of \eqref{eqn:addboxes} appropriately.

\end{proof}
\section{Bivariate generating functions: \texorpdfstring{$\boldsymbol{2}$}{2}-rowed case}
\label{sec:conjclosedforms}

The aim of this section is to prove the following theorem for the bivariate generating functions of $2$-rowed tight cylindric partitions.
\begin{thm}
\label{thm:bivar2rowtight}
Fix $\ell\geq 1$. Let $0\leq b\leq \lfloor \frac{\ell}{2}\rfloor$. We have:
\begin{align*}
	T_{(\ell-b,b)}&(z,q) = T_{(b,\ell-b)}(z,q)\\
	&=\sum_{n_1,\cdots, n_\ell\geq 0}
	\dfrac{z^{N_1}q^{\frac{1}{2}(N_1^2+\cdots +N_\ell^2)\, +\, \frac{1}{2}(n_1+n_3+\cdots n_{2b-1})\, +\,\frac{1}{2}(N_{2b+1}+N_{2b+2}+\cdots+N_{\ell})}(q)_{N_1} }{(zq;q)_{N_1}(q)_{n_1}\cdots (q)_{n_\ell}}\\
    &=\sum_{N_1,\cdots, N_\ell\geq 0}
	\dfrac{z^{N_1}q^{ \sum_{i=1}^\ell {\binom{N_i+1}{2}}\, - \, \sum_{i=1}^b N_{2i}}(q)_{N_1} }{(zq;q)_{N_1}(q)_{N_1-N_2}\cdots (q)_{N_{\ell-1}-N_{\ell}}(q)_{N_\ell}}.
\end{align*}
where, as usual, $N_j=n_j+n_{j+1}+\cdots+n_{\ell}$.
\end{thm}
With $z\mapsto 1$ and $\ell$ odd, this corresponds to the sum in \cite[Thm.\ 1.4]{KimYee-parity}, and for  $z\mapsto 1$ and $\ell$ even, this corresponds to \cite[Eq.\ (1.9)]{And-parity}. (Note that dilating $q^2\mapsto q$ in their sums is necessary to match ours, as noted by~\cite{CheEtAl}.) Kur\c{s}ung\"oz also considered these sums~\cite{Kur-parity}. Finally, see Corollary 1.5(a) of Stembridge~\cite{Stem} for the case with $z \mapsto 1$ and $b=0$.

\subsection{Preliminaries}
Fix $\ell\ge1$. First, define for a vector $\vv=(v_1,\cdots,v_\ell)$
\begin{align}
S(t;\vv;z,q) &= \sum_{n_1,\dots,n_\ell \ge 0} \frac{z^{N_1} q^{\left(N_1^2 + N_2^2 + \cdots + N_\ell^2 + N_1 + N_2 + \cdots + N_\ell\right)/2 +\vv\cdot\vn}(q;q)_{N_1}}{(zq;q)_{N_1+t}(q;q)_{n_1}(q;q)_{n_2} \cdots (q;q)_{n_\ell}}.
\end{align}
We will frequently suppress $z$ and $q$ and simply write  $S(t;\vv)$.

We let $e_j$ be the standard basis vector with a $1$ in the $j^{\text{th}}$ position and $0$s everywhere else. 
We shall use the following vectors:
\begin{align*}
\delta_j&=e_j+e_{j+1}+\cdots+e_{\ell},\\
\Delta_i &= \delta_1 + \delta_2+\cdots+\delta_{i},\\
\eta_i &= \delta_2 + \delta_4+\cdots+\delta_{2i}.
\end{align*}
If $j>\ell$, then $e_j=\delta_j=0$.
Note that $e_j=\delta_j-\delta_{j+1}$.

By symmetry, $T_{(\ell-b,b)}(z,q) = T_{(b,\ell-b)}(z,q)$. Thus, it suffices to prove Theorem~\ref{thm:bivar2rowtight} by showing that the appropriate set of multisums satisfy half of the ``diamond relations''. Writing $R_{b}(z,q)=T_{(\ell-b,b)}(z,q) = T_{(b,\ell-b)}(z,q)$ for $0 \le b \le \lfloor \ell/2\rfloor$, we restate Theorem~\ref{thm:bivar2rowtight} (suppressing the $q$ argument of $R_i$) as:
\begin{thm} \label{thm:maindiamond}
A solution to
\begin{align}
R_0(z) + \dfrac{zq^2}{1-zq^2}R_{0}(zq^2) &= \dfrac{R_{1}(zq)}{1-zq} \\
R_i(z) + \dfrac{zq^2}{1-zq^2}R_{i}(zq^2) &= \dfrac{R_{i-1}(zq)}{1-zq}+\dfrac{R_{i+1}(zq)}{1-zq}
 -\dfrac{R_i(zq^2)}{1-zq^2}, \quad 1\leq i < \frac{\ell-1}2 \\
R_{\ell/2}(z) + \dfrac{zq^2}{1-zq^2}R_{\ell/2}(zq^2) &= \dfrac{2R_{\ell/2-1}(zq)}{1-zq} -\dfrac{R_{\ell/2}(zq^2)}{1-zq^2}, \quad \text{if $\ell$ is even,}
\end{align}
\begin{align}
R_{(\ell-1)/2}(z) &+ \dfrac{zq^2}{1-zq^2}R_{(\ell-1)/2}(zq^2) \notag\\
&=  \dfrac{R_{(\ell-3)/2}(zq)}{1-zq}+\dfrac{R_{(\ell-1)/2}(zq)}{1-zq}-\dfrac{R_{(\ell-1)/2}(zq^2)}{1-zq^2}, \quad \text{if $\ell$ is odd},
\end{align}
with initial conditions $R_{i}(0,q)=R_i(z,0)=1$, is given by
\begin{align} \label{eqn:Rizdefn}
    R_i(z,q) = S(0;-\eta_i;z,q).
\end{align}
 for $0\leq i\leq \left\lfloor{\ell}/{2}\right\rfloor$.
\end{thm}

We have the following relations for $1\leq j\leq \ell$:
\begin{align}
\rel_j(t;\vv) &: S(t;\vv) - S(t;\vv+e_j) - zq^{v_j+j}S(t+1;\vv+ \Delta_j)\notag\\
&\qquad+zq^{v_j+j+1}S(t+1;\vv+\delta_1+ \Delta_j)=0.
\end{align}
This can be obtained routinely by multiplying each summand of $S(t;\vv)$ by $\left(1-q^{n_j}\right)$, noting that summands corresponding to $n_j=0$ vanish, and shifting the summation index $n_j\mapsto n_j+1$. 
In addition, we have:
\begin{align}
\rel_0(t;\vv): S(t;\vv)-S(t+1;\vv)+zq^{t+1}S(t+1;\vv+\delta_1)=0;
\end{align}
obtained easily by subtracting the corresponding summands of $S(t;\vv)$ and $S(t+1;\vv)$.

\subsection{Proof of Theorem~\ref{thm:maindiamond}}
We begin with the following proposition:
\begin{prop}
For $1 \le i <\ell/2,$ we have:
\begin{align} \label{eqn:Sfourterm}
S(1;\delta_1 -\eta_{i-1})- S(1;2\delta_1-\eta_i)-S(1;e_1+e_2-\eta_i)+S(1;\delta_1-\eta_{i+1}) = 0.
\end{align}
Furthermore, if $\ell$ is even, then
\begin{equation}  \label{eqn:Sfourtermevenmiddle}
2S(1;\delta_1 -\eta_{{\ell/2}-1})- S(1;2\delta_1-\eta_{\ell/2})-S(1;e_1+e_2-\eta_{\ell/2})= 0.
\end{equation}
\end{prop}

\begin{proof}
First, let  $1 \le j \le \ell-2$, and suppose that $\vv$ is a vector such that $\vv_{j}=\vv_{j+2}+1$. Then,
\begin{align}
&\rel_j\left(1;\vv+\delta_{j+1}\right)-\rel_{j+2}\left(1;\vv-\delta_{j+2}\right) \\
&=S(1;\vv+\delta_{j+1}) - S(1;\vv+\delta_{j+1}+e_j) - zq^{v_j+j}S(2;\vv+\delta_{j+1}+ \Delta_j)\\
&\qquad +zq^{v_j+j+1}S(2;\vv+\delta_{j+1}+\delta_1+ \Delta_j) \\
&\,\,-\left(S(1;\vv-\delta_{j+2}) - S(1;\vv-\delta_{j+2}+e_{j+2}) - zq^{v_{j+2}-1+j+2}S(2;\vv-\delta_{j+2}+ \Delta_{j+2})\right.\\
&\qquad \left.+zq^{v_{j+2}+j+3}S(2;\vv-\delta_{j+2}+\delta_1+ \Delta_{j+2})\right) \\
&=S(1;\vv+\delta_{j+1}) - S(1;\vv+\delta_{j}) - zq^{v_j+j}S(2;\vv+ \Delta_{j+1})\\
&\qquad+zq^{v_j+j+1}S(2;\vv+\delta_1+ \Delta_{j+1}) \\
&\,\,-S(1;\vv-\delta_{j+2}) + S(1;\vv-\delta_{j+3}) +zq^{v_{j}+j}S(2;\vv+\Delta_{j+1})\\
&\qquad-zq^{v_{j}+j+1}S(2;\vv+\delta_1+ \Delta_{j+1}) \\
&=S(1;\vv+\delta_{j+1}) - S(1;\vv+\delta_{j}) -S(1;\vv-\delta_{j+2}) + S(1;\vv-\delta_{j+3}) =0. \label{eq:newfour}
\end{align}

Now, let $1 \le i <\ell/2$. 
Noting that $\delta_1-\eta_i$ has  the property that  $\vv_{j}=\vv_{j+2}+1$ for $1\le j\le 2i-1$, 
consider $\ds\sum_{j=1}^{2i-1} \left(\rel_j\left(1;\delta_1-\eta_i+\delta_{j+1}\right)-\rel_{j+2}\left(1;\delta_1-\eta_i-\delta_{j+2}\right)\right)$. Using~\eqref{eq:newfour}, this expands as
\begin{align*}
 \sum_{j=1}^{2i-1} &\left( S(1;\delta_1-\eta_i+\delta_{j+1}) - S(1;\delta_1-\eta_i+\delta_{j}) - S(1;\delta_1-\eta_i-\delta_{j+2})
 \right. \notag \\
  &\left.+ S(1;\delta_1-\eta_i-\delta_{j+3})\right).
\end{align*}
The first and second terms of the series above telescope, as do the third and fourth, thus providing us with
\begin{align*}
&S(1;\delta_1-\eta_i+\delta_{2i}) - S(1;\delta_1-\eta_i+\delta_{1}) - S(1;\delta_1-\eta_i-\delta_{3}) + S(1;\delta_1-\eta_i-\delta_{2i+2}) \\
&=S(1;\delta_1-\eta_{i-1}) - S(1;2\delta_1-\eta_i) - S(1;e_1+e_2-\eta_i) + S(1;\delta_1-\eta_{i+1}) = 0,
\end{align*}
proving~\eqref{eqn:Sfourterm}. To prove~\eqref{eqn:Sfourtermevenmiddle}, instead consider (assuming $\ell$ is even)
\begin{align*} \sum_{j=1}^{\ell-2} &\left(\rel_j\left(1;\delta_1-\eta_{\ell/2}+\delta_{j+1}\right)-\rel_{j+2}\left(1;\delta_1-\eta_{\ell/2}-\delta_{j+2}\right)\right) \\
&\quad +\rel_{\ell-1}\left(1;\delta_1-\eta_{\ell/2}+\delta_{\ell}\right)-\rel_{\ell}\left(1;\delta_1-\eta_{\ell/2}\right).
\end{align*}
This expands as
\begin{align*}
 &\sum_{j=1}^{\ell-2} \! \left( S(1;\delta_1 \!- \!\eta_{\frac{\ell}2} \!+ \!\delta_{j+1}) \! - \! S(1;\delta_1 \!- \!\eta_{\frac{\ell}2} \!+ \!\delta_{j}) \! - \! S(1;\delta_1 \!- \!\eta_{\frac{\ell}2} \!- \!\delta_{j+2})  \!+ \! S(1;\delta_1 \!- \!\eta_{\frac{\ell}2} \!- \!\delta_{j+3})\right) \\
 & +\! S(1;\delta_1\!-\!\eta_{\frac\ell 2}\!+\!\delta_{\ell})\! -\! S(1;\delta_1\!-\!\eta_{\frac\ell 2}\!+\!\delta_{\ell}\!+\!e_{\ell-1})\!
 -\!zq^{1-(\ell/2-1)+\ell-1}S(2;\delta_1\!-\!\eta_{\frac\ell 2}\!+\!\delta_{\ell}\!+\!\Delta_{\ell-1}) \\
 &+zq^{1-(\ell/2-1)+\ell-1+1}S(2;2\delta_1-\eta_{\frac\ell 2}+\delta_{\ell}+\Delta_{\ell-1}) \\
&-\left(S(1;\delta_1-\eta_{\ell/2}) - S(1;\delta_1-\eta_{\ell/2}+e_{\ell})
-zq^{1-\ell/2+\ell}S(2;\delta_1-\eta_{\ell/2}+\Delta_{\ell})\right. \\
&\quad\quad \left.
+zq^{1-\ell/2+\ell+1}
S(2;2\delta_1-\eta_{\ell/2}+\Delta_{\ell})\right) \\
&= S(1;\delta_1-\eta_{\frac{\ell}2}+\delta_{\ell-1}) - S(1;\delta_1-\eta_{\frac{\ell}2}+\delta_{1}) -S(1;\delta_1-\eta_{\frac{\ell}2}-\delta_{3})+ S(1;\delta_1-\eta_{\frac{\ell}2}-\delta_{\ell+1}) \\
 & \quad \quad +
 S(1;\delta_1-\eta_{\frac{\ell}2}+\delta_{\ell}) - S(1;\delta_1-\eta_{\frac{\ell}2}+\delta_{\ell-1})-S(1;\delta_1-\eta_{\frac{\ell}2}) + S(1;\delta_1-\eta_{\frac{\ell}2}+e_{\ell}) \\
&=  - S(1;2\delta_1-\eta_{\ell/2}) -S(1;\delta_1-\eta_{\ell/2}-\delta_{3}) + 2S(1;\delta_1-\eta_{\ell/2-1}) \\ &=0.
\end{align*} Above, we had to recall that $\delta_{\ell+1}=\vz$ and that $e_\ell=\delta_\ell$. The last equality is equivalent to ~\eqref{eqn:Sfourtermevenmiddle}.
\end{proof}
We are now ready to complete the proof of Theorem~\ref{thm:maindiamond}.

\begin{proof}
First, it is easy to see that $S(0;-\eta_i;0,q)=S(0;-\eta_i;z,0)=1$ for $0\leq i\leq \left\lfloor\frac{\ell}{2}\right\rfloor$.

Now, observe that $\ds \frac{S(0;\vv;zq)}{1-zq}=S(1;\vv+\delta_1;z)$ and 
$\ds \frac{S\left(0;\vv;zq^2\right)}{(1-zq)\left(1-zq^2\right)}=S(2;\vv+2\delta_1;z)$. Substituting~\eqref{eqn:Rizdefn} into the ``diamond relations'' produces:
\begin{align} \label{eqn:edge}
 S(0;\vz)+ zq^2(1\!-\!zq)S(2;2\delta_1) &= S(1;e_1) \\ \notag
S(0;-\eta_i) + (1\!-\!zq)(1\!+\!zq^2)S(2; 2\delta_1 -\eta_i) &= S(1;\delta_1 -\eta_{i-1})+S(1;\delta_1-\eta_{i+1}), \\ & 1\leq i < \frac{\ell-1}2 \label{eqn:awayfrommiddle} \\
\notag
S(0;-\eta_{\frac{\ell-1}2})+ (1\!-\!zq)(1\!+\!zq^2) S(2; 2\delta_1 -\eta_{\frac{\ell-1}2}) &=  S(1;\delta_1-\eta_{\frac{\ell-3}2})+S(1;\delta_1-\eta_{\frac{\ell-1}2}), \\ & \text{if $\ell$ is odd}  \label{eqn:middleodd}\\
S(0;-\eta_{\ell/2}) + (1\!-\!zq)(1\!+\!zq^2) S(2; 2\delta_1 -\eta_{\ell/2}) &=2S(1;\delta_1-\eta_{\ell/2-1}), \notag \\
&\text{if $\ell$ is even} \label{eqn:middleeven} 
\end{align}
    
Consider the linear combination of relations given by
$\rel_0(0;\vz) +\rel_1(1;\vz)-zq\,\rel_0(1;\delta_1)$.
Expanding this out gives:
\begin{align*}
&S(0;\vz)-S(1;\vz)+zq\,S(1;\delta_1) \\
&+S(1;\vz) - S(1;e_1) - zq\,S(2;\delta_1)+zq^{2}S(2;2\delta_1) \\
&-zq\left(S(1;\delta_1)-S(2;\delta_1)+zq^2\,S(2;2\delta_1)\right) \\
&=  S(0;\vz)- S(1;e_1) +zq^{2}S(2;2\delta_1)  -z^2q^3\,S(2;2\delta_1) &=0.
\end{align*}
Thus, we have verified that~\eqref{eqn:edge} is true.

Now, for $1\leq i < {\ell}/{2}$, consider the linear combination of relations given by 
\begin{equation*}
\rel_0(0;-\eta_i)
-\rel_0(1;2\delta_1-\eta_i)
-zq\, \rel_0(1;\delta_1-\eta_i)
+\rel_1(1;-\eta_i)
+\rel_2(1;e_1-\eta_i).
\end{equation*}
Expanding this out gives:
\begin{align*}
&S(0;-\eta_i)-\underline{S(1;-\eta_i)}
+\underline{zqS(1;-\eta_i+\delta_1)}
\\
&-S(1;2\delta_1-\eta_i)+S(2;2\delta_1-\eta_i)-\underline{zq^2S(2;3\delta_1-\eta_i)}
\\
& -\underline{zqS(1;\delta_1-\eta_i)}
+\underline{zqS(2;\delta_1-\eta_i)}-z^2q^3S(2;2\delta_1-\eta_i)
\\
&
+\underline{S(1;-\eta_i)}-\underline{S(1;-\eta_i+e_1)}-\underline{zqS(2;-\eta_i+\delta_1)}
+zq^2S(2;2\delta_1-\eta_i)
\\
&+\underline{S(1;e_1-\eta_i)}-S(1;e_1+e_2-\eta_i)-zqS(2;e_1-\eta_i+\delta_1+\delta_2) \\ & \quad + \underline{zq^2S(2;e_1-\eta_i+2\delta_1+\delta_2)} = 0.
\end{align*}
We have underlined the terms that cancel (recall that $e_1=\delta_1-\delta_2$). Considering the uncancelled terms and factoring the coefficient of $S(2;2\delta_1-\eta_i)$ allows us to conclude that
\begin{align*}
S(0;-\eta_i) +(1-zq)(1+zq^2)S(2;2\delta_1\!-\eta_i)- S(1;2\delta_1\!-\eta_i)-S(1;e_1+e_2-\eta_i) = 0.
\end{align*}
If we subtract~\eqref{eqn:Sfourterm} from this, we obtain
\begin{align*}
S(0;-\eta_i) +(1-zq)(1+zq^2)S(2;2\delta_1-\eta_i)
-S(1;\delta_1 -\eta_{i-1})
-S(1;\delta_1-\eta_{i+1})
= 0
\end{align*}
which is exactly~\eqref{eqn:awayfrommiddle} (after rearrangement).

Suppose $\ell$ is odd. Consider the linear combination of relations given by 
\begin{align*}
&\rel_0(0;-\eta_{(\ell-1)/2})
-\rel_0(1;2\delta_1-\eta_{(\ell-1)/2})
-zq\, \rel_0(1;\delta_1-\eta_{(\ell-1)/2}) \\
& \quad +\rel_1(1;-\eta_{(\ell-1)/2})
+\rel_2(1;e_1-\eta_{(\ell-1)/2}).
\end{align*}
Expanding this out gives:
\[
\begin{aligned}
&\,S(0;-\eta_{(\ell-1)/2})-\underline{S(1;-\eta_{(\ell-1)/2})}
+\underline{zqS(1;-\eta_{(\ell-1)/2}+\delta_1)}
\notag\\
&-S(1;2\delta_1-\eta_{(\ell-1)/2})+S(2;2\delta_1-\eta_{(\ell-1)/2})-\underline{zq^2S(2;3\delta_1-\eta_{(\ell-1)/2})}
\notag\\
& -\underline{zqS(1;\delta_1-\eta_{(\ell-1)/2})}
+\underline{zqS(2;\delta_1-\eta_{(\ell-1)/2})}-z^2q^3S(2;2\delta_1-\eta_{(\ell-1)/2})
\notag\\
&
+\underline{S(1;-\eta_{(\ell-1)/2})}-\underline{S(1;-\eta_{(\ell-1)/2}+e_1)}-\underline{zqS(2;-\eta_{(\ell-1)/2}+\delta_1)}
\notag\\
&\qquad +zq^2S(2;2\delta_1-\eta_{(\ell-1)/2})
\notag\\
&+\underline{S(1;e_1-\eta_{(\ell-1)/2})}-S(1;e_1+e_2-\eta_{(\ell-1)/2})-zqS(2;e_1-\eta_{(\ell-1)/2}+\delta_1+\delta_2)  \\
& \qquad  + \underline{zq^2S(2;e_1-\eta_{(\ell-1)/2}+2\delta_1+\delta_2)}.
\end{aligned}
\]
Again, collecting the uncancelled terms, and recalling $e_1+\delta_2=\delta_1$, gives
\begin{equation*}
S(0;-\eta_{\frac{\ell-1}2})
-S(1;2\delta_1-\eta_{\frac{\ell-1}2}) -S(1;e_1+e_2-\eta_{\frac{\ell-1}2}) 
+\left(1\!-\!zq\right)\left(1\!+\!zq^2\right)S(2;2\delta_1-\eta_{\frac{\ell-1}2}) =0
\end{equation*}
Now subtract the $i=\frac{\ell-1}2$ case of~\eqref{eqn:Sfourterm} to obtain
\begin{equation*}
S(0;-\eta_{\frac{\ell-1}2})
-S(1;\delta_1 -\eta_{\frac{\ell-3}2})
-S(1;\delta_1-\eta_{\frac{\ell-1}2+1})
+\left(1-zq\right)\left(1+zq^2\right)S(2;2\delta_1-\eta_{\frac{\ell-1}2}) =0,
\end{equation*}
and using the fact that $\eta_{{(\ell-1)/2}+1}=\eta_{{(\ell-1)/2}}$ gives us
\begin{equation*}
S(0;-\eta_{\frac{\ell-1}2})
-S(1;\delta_1 -\eta_{{\frac{\ell-3}2}})
-S(1;\delta_1-\eta_{{\frac{\ell-1}2}})
+\left(1-zq\right)\left(1+zq^2\right)S(2;2\delta_1-\eta_{\frac{\ell-1}2}) =0,
\end{equation*}
which is exactly~\eqref{eqn:middleodd}.

Finally, suppose $\ell$ is even. Consider the linear combination of relations given by 
\begin{equation*}
\rel_0(0;-\eta_{\ell/2})
-\rel_0(1;2\delta_1\!-\eta_{\ell/2})
-zq \rel_0(1;\delta_1\!-\eta_{\ell/2})
+\rel_1(1;-\eta_{\ell/2})
+\rel_2(1;e_1-\eta_{\ell/2}).
\end{equation*}
Expanding this out gives:
\[
\begin{aligned}
&\,S(0;-\eta_{\ell/2})-\underline{S(1;-\eta_{\ell/2})}
+\underline{zqS(1;-\eta_{\ell/2}+\delta_1)}
\notag\\
&-S(1;2\delta_1-\eta_{\ell/2})+S(2;2\delta_1-\eta_{\ell/2})-\underline{zq^2S(2;3\delta_1-\eta_{\ell/2})}
\notag\\
& -\underline{zqS(1;\delta_1-\eta_{\ell/2})}
+\underline{zqS(2;\delta_1-\eta_{\ell/2})}-z^2q^3S(2;2\delta_1-\eta_{\ell/2})
\notag\\
&
+\underline{S(1;-\eta_{\ell/2})}-\underline{S(1;-\eta_{\ell/2}+e_1)}-\underline{zqS(2;-\eta_{\ell/2}+\delta_1)}
+zq^2S(2;2\delta_1-\eta_{\ell/2})
\notag\\
&+\underline{S(1;e_1-\eta_{\ell/2})}-S(1;e_1+e_2-\eta_{\ell/2})-zqS(2;e_1-\eta_{\ell/2}+\delta_1+\delta_2)  \\
&+ \underline{zq^2S(2;e_1-\eta_{\ell/2}+2\delta_1+\delta_2)}.
\end{aligned}
\]
Cancelling as before gives
\begin{equation*}
S(0;-\eta_{\ell/2})
-S(1;2\delta_1\!-\eta_{\ell/2}) -S(1;e_1+e_2-\eta_{\ell/2}) 
+(1-zq)\left(1\!+\!zq^2\right)S(2;2\delta_1\!-\eta_{\ell/2}) = 0.
\end{equation*}
Now, subtracting~\eqref{eqn:Sfourtermevenmiddle} gives the following equation, equivalent to ~\eqref{eqn:middleeven}.
\begin{equation*}
S(0;-\eta_{\ell/2})
-2S(1;\delta_1 -\eta_{{\ell/2}-1})
+(1-zq)\left(1+zq^2\right)S(2;2\delta_1-\eta_{\ell/2})
= 0.
\end{equation*}
\end{proof}

\begin{rem}
\label{rem:allcyl2rowzq}
Here we record the analogue of Theorem \ref{thm:bivar2rowtight} for \emph{all} cylindric partitions of a given two-rowed profile.
Let $\ell\geq 1$, $0\leq b \leq \lfloor\frac{\ell}{2}\rfloor=k$. If $\ell$ is odd, let $s=1$, and $s=2$ otherwise.
Then:
\begin{align}
&C_{(\ell-b,b)}(z,q)=C_{(b,\ell-b)}(z,q)\notag\\
&=\dfrac{1}{(zq)_\infty}
\sum_{N_1\geq N_2\geq \cdots \geq N_k\geq 0 }
\dfrac{z^{N_1}q^{N_1^2+N_2^2+\cdots+N_k^2+N_{b+1}+N_{b+2}+\cdots+N_{k}}}{(q)_{N_1-N_2}(q)_{N_2-N_3}\cdots 
(q)_{N_{k-1}-N_k}
{(q^s;q^s)_{N_k}}}.
\label{eq:allcyl2rowzq}
\end{align}
This follows straightforwardly from the results of Foda and Welsh~\cite[Sec.\ 7, 9]{FodWel}. 
Alternately, one may use $L\rightarrow\infty$ limits of the finitizations of these generating functions found by Warnaar \cite[Eq.\ (7.24), (7.25)]{War-AndGorCyl}.
\end{rem}

\section{Bivariate generating functions: \texorpdfstring{$r$}{r} rows, level \texorpdfstring{$1$}{1}}
Going beyond the $2$-rowed case, we have the following result.
\begin{prop}
\label{prop:rrowlev1}
For $c=(c_1,\cdots, c_r)$ a profile of level $1$, i.e., $c_1+\cdots+c_r=1$,
\begin{align}
T_c(z;q) = 1+\sum_{n\geq 1}z^n\left(\dfrac{(q^r;q^r)_n}{(q)_n} - \dfrac{(q^r;q^r)_{n-1}}{(q)_{n-1}} \right).
\label{eqn:rrowlev1}
\end{align}
\end{prop}
\begin{proof}
Due to circular symmetry, we may assume that $c=(1,0,\cdots,0)$. Let $\pi=(\pi^{(1)}_1,\cdots, \pi^{(r)}_1)\in \mathscr{T}_c$, with $\max(\pi)=n (\geq 1)$.
We shall read $\pi$ column-wise, i.e., 
$ \pi^{(1)}_1 \geq \pi^{(2)}_1 \geq \cdots \geq \pi^{(r)}_1 
\boldsymbol{\geq } \pi^{(1)}_2 \geq \pi^{(2)}_2 \geq\cdots \geq\pi^{(r)}_2 
\boldsymbol{\geq } \pi^{(1)}_3 \geq \cdots$. Due to the profile, note that $\pi^{(r)}_1$ is stacked on top of $\pi^{(1)}_2$, and thus we have $\pi^{(r)}_1\geq \pi^{(1)}_2$ and so on. The part $\max(\pi)=n$ can only occur as $\pi^{(1)}_1,\cdots, \pi^{(j)}_1$ for some $1\leq j<r$, as $j= r$ breaks tightness. Continuing, the next largest part can also occur at most $r-1$ times. 
Therefore, the generating fun ction for tight cylindric partitions with largest part $n (\geq 1)$ is as below; easily seen to equal the corresponding summand in \eqref{eqn:rrowlev1}:
\begin{align}
z^n(q^n+\cdots+q^{(r-1)n})\prod_{1\leq j\leq n-1}(1+q^{j}+\cdots+q^{(r-1)j}).
\end{align}
\end{proof}

\section{Dousse--Hardiman--Konan partitions}\label{sec:DHKpartitions}

\subsection{The basics}
In \cite{DouHarKon}, Dousse, Hardiman, and Konan defined a class of colored partitions, connected them to crystal bases for standard $\ala{sl}_2$ modules, and thus proved a sum-to-product identity encapsulated in Theorem \ref{thm:dhk} below. The aim of this section is to give two alternate proofs of this theorem using $2$-rowed tight cylindric partitions.

\begin{defi}
Let $\ell\geq 1$ and $0\leq a\leq \ell$. Let $\mathrm{DHK}_{a,\ell}$ be the set of $\ell+1$ colored partitions $\lambda_1\geq \lambda_2\geq \cdots\geq \lambda_s=0$, with allowed colors being $0, 1,\cdots, \ell$, satisfying:
\begin{enumerate}
	\item $	\lambda_i-\lambda_{i+1} = |u_i-u_{i+1}|$ for all $i$ where $u_i$ is the color of the part $\lambda_i$,
	\item the color of $\lambda_s=0$ is $a$,
	\item there is exactly one part of size $0$.
\end{enumerate} 
We ignore the last $0$ part while counting the number of parts, i.e., the number of parts of $\lambda_1\geq \lambda_2\ge\cdots\geq \lambda_s=0$ is $s-1$. The symbol $\#$ denotes the number of parts.
We use $D_{a,\ell}(z,q)$ to denote the bivariate generating functions of DHK partitions:
\begin{align}
D_{a,\ell}(z,q) = \sum_{\lambda \in \mathrm{DHK}_{a,\ell}}  z^{\#(\lambda)}q^{\mathrm{wt}(\lambda)}.
\end{align}
Note that the entire partition is uniquely determined solely by its color sequence and the initial part $0$.
\end{defi}
\begin{ex}
Let $\lambda=8_{\textcolor{blue}{2}} + 7_{\textcolor{blue}{3}} + 7_{\textcolor{blue}{3}} + 5_{\textcolor{blue}{1}} + 4_{\textcolor{blue}{0}} + 3_{\textcolor{blue}{1}} + 3_{\textcolor{blue}{1}} + 3_{\textcolor{blue}{1}} + 2_{\textcolor{blue}{0}} + 1_{\textcolor{blue}{1}} + 0_{\textcolor{blue}{0}}$ (here, subscripts are colors).
Then, $\lambda\in \mathrm{DHK}_{0,3}$, $\#(\lambda)=10$, $\mathrm{wt}(\lambda)=43$.
\end{ex}

In \cite{DouHarKon}, using crystal bases for level $\ell$ standard modules for $\ala{sl}_2$, it was proved that:
\begin{thm}[Eq.\ (1.3) \cite{DouHarKon}]
	\label{thm:dhk}
	For $\ell\geq 1$, $0\leq a\leq \ell$, we have:
	\begin{align}
		D_{a,\ell}(1,q) = \dfrac{(q^{a+1}, q^{\ell-a+1},q^{\ell+2};\,q^{\ell+2} )}{(q;q^2)_\infty (q;q)_\infty} = \mathrm{Pr}\left(L_{\ala{sl}_2}((\ell-a)\Lambda_0+a\Lambda_1)\right).
	\end{align}
\end{thm}

\subsection{Proof using recurrences}
\begin{prop}
	\label{prop:dhkrec}
    We have, for $\ell\geq 1$ and $0\leq a \leq \ell$, the following:
\begin{align}
D_{a,\ell}(z,q) &= 1 + \sum_{\substack{0\leq j\leq \ell\\ j\neq a}} \dfrac{zq^{|j-a|}}{1-zq^{|j-a|}}D_{j,\ell}(zq^{|j-a|},q),
\label{eqn:dhkrec}
\\
D_{a,\ell}(0,q) = D_{a,\ell}(z,0)&=1.
\end{align}
\end{prop}
\begin{proof}
Discard the partition that has a single part $0$ with color $a$. This contributes the $1$ on the right side of \eqref{eqn:dhkrec}. 
The remaining partitions have a smallest part that is not $0$ and this part can have any color from $0$ to $\ell$ except $a$, since having the color $a$ would force two $0$s to occur. 
Thus, the partitions start as: $\cdots 
+\underbrace{|j-a|_{j} + |j-a|_j+ \cdots + |j-a|_j}_{t\,\, \mathrm{times}}+0_a$ for some $j\in\{0,1,\cdots,\ell\}\backslash \{a\}$, $t\geq 1$.
Such partitions can be obtained by taking partitions counted by $D_{j,\ell}$, deleting the initial $0_j$ part, adding $|j-a|_j$ to each part, then attaching a string of $t$ parts, each of size $|j-a|_j$, and finally attaching a $0_a$ part. This contributes $zq^{|j-a|t}D_{j,\ell}(zq^{|j-a|},q)$ to the right side of \eqref{eqn:dhkrec}. Summing over $t\geq 1$ and $j\in \{0,1,\cdots,\ell\}\backslash \{a\}$ now gives us the summation in the right side of \eqref{eqn:dhkrec}.

\end{proof}

It is now clear that one can prove Theorem \ref{thm:dhk} by noting that $\mathrm{DHK}$ partitions and the corresponding tight cylindric partitions have the same bivariate generating functions, being the unique solutions to the same set of functional equations given in Propositions \ref{prop:2rowtightrec} and \ref{prop:dhkrec}. Since the univariate generating functions of tight-cylindric partitions have the product formula as in  Theorem \ref{thm:tightprod}, so do the $\mathrm{DHK}$ partitions.

\subsection{A bijection with tight cylindric partitions}

This proof is not at all surprising as both (tight) cylindric partitions and the $\mathrm{DHK}$ partitions arise from crystal bases.

We associate a ordinary partition to a $2$-string abacus of type $\mathcal{A}_{\ell-a,a}$.  Starting at the zeroth yoke, record the total number of vacancies to the left of the $i^{\text{th}}$ yoke. Continuing on with our example from the end of section 2, the total number of vacancies are $ 0,1,2,3,3,3,4,5,7,7,8$. (This can be alternately obtained by ``flattening'' the $2$-rowed cylindric partition, taking the conjugate, and adjoining the ``ground'' $0$.) Naturally, the number of non-zero entries is the number of (non-zero) yokes.

Next, we record the shape of each yoke, i.e., how far the top bead is shifted with respect to the bottom bead. The shape sequence, starting at the zeroth yoke, for our working example is thus
$0, 1,0,1,1,1,0,1,3,3,2$.

Combining these two sequences in a colored partition, we obtain:
$0_{\textcolor{blue}{0}} + 1_{\textcolor{blue}{1}} + 2_{\textcolor{blue}{0}} + 3_{\textcolor{blue}{1}} + 3_{\textcolor{blue}{1}}+ 3_{\textcolor{blue}{1}} + 4_{\textcolor{blue}{0}} + 5_{\textcolor{blue}{1}} + 7_{\textcolor{blue}{3}}+7_{\textcolor{blue}{3}}+8_{\textcolor{blue}{2}}$.

It can be easily checked that the number of vacancies between two adjacent yokes of shapes $\alpha$ and $\beta$ is $|\alpha-\beta|$ (see \cite[Eq.\ (31)]{FodWel}), assuming that the abacus is tight. 
Reversing the resulting colored sequence of non-negative integers, we have thus obtained a DHK partition of class $\mathrm{DHK}_{a,\ell}$.

It is easily checked that this process is reversible, starting from a DHK partition: just place the yokes of required types given by colors and then count beads to the right of any vacancy on each row.

\noindent	\textbf{Summary of the proof}: Starting with a tight $2$-rowed cylindric partition of profile $(\ell-a,a)$, we  obtain a tight $2$-string abacus of type $\mathcal{A}_{\ell-a,a}$, from which we produce a colored partition of type $\mathrm{DHK}_{a,\ell}$.  This is the required bijection, and it takes the maximum part in the tight cylindric partition to the number of (non-zero) parts in the $\mathrm{DHK}$ partition.

\section{Questions}

Some natural questions and directions for further investigation are collected here.

\begin{enumerate}[leftmargin=*]
\item It would be highly desirable to have a direct proof of Theorem \ref{thm:bivar2rowtight}.
\item To the best of our knowledge, nothing seems to be known about the bivariate generating functions in the tight $r$-rowed case with $r\geq 3$, 
except of course the level $1$ case of Proposition \ref{prop:rrowlev1}.
\item With $z\mapsto 1$, the sums in Theorem \ref{thm:bivar2rowtight} are the same as in the theorems of Andrews \cite{And-parity} and Kim--Yee \cite{KimYee-parity}. These theorems are in turn are related to Kleshchev multipartitions for $\ala{sl}_2$ \cite{CheEtAl} (see also \cite{Tsu-RRKleshchev}, 
\cite{FodLecOkaThiWel}). Is there a direct bijection relating Kleshchev multipartitions and tight cylindric partitions?
\item For both tight and not necessarily tight cylindric partitions, computer experiments suggest that the polynomial in $z$ for each coefficient of $q^n$ has the following property: The coefficients first rise strictly, then the maximum is reached (possibly at one or two adjacent powers of $z$) and then the coefficients decrease strictly. It would be interesting to prove this.
\item In terms of representation theory,  what does the ``maximum part'' statistic mean? 
\item Are there analogs of other results about cylindric partitions (say, those of Gessel and Krattenthaler~\cite{GesKra}) for tight cylindric partitions? Kur\c sung\"oz and \"Omr\"uuzun Seyrek~\cite{KS} showed that cylindric partitions (of a particular profile) are in bijection with ordinary partitions and colored partitions into distinct parts (with certain restrictions); is there an analog for tight cylindric partitions?
\end{enumerate}


\end{document}